\documentclass[12pt]{amsart}
\usepackage{amssymb}
\usepackage{amsthm}
\usepackage{amsmath}
\usepackage[T1]{fontenc}
\usepackage[latin1]{inputenc}
\usepackage[all]{xy}
\usepackage{verbatim}
\usepackage{vmargin}
\usepackage{multirow}
\usepackage{array}
\setpapersize{USletter}
\setmargrb{1in}{1in}{1in}{1in}
\DeclareMathOperator{\Span}{span}
\DeclareMathOperator{\Sing}{sing}

\DeclareMathOperator{\rk}{rk}

\DeclareMathOperator{\Pic}{Pic}

\DeclareMathOperator{\codim}{codim}

\newtheorem{thm}{Theorem}[section]
\newtheorem{prop}[thm]{Proposition}
\newtheorem{exa}[thm]{Example}
\newtheorem{defn}[thm]{Definition}
\newtheorem{claim}[thm]{Claim}

\newtheorem{rem}[thm]{Remark}
\newtheorem{cor}[thm]{Corollary}

\newtheorem*{thm1}{Theorem \ref{idealSV}}

\title{The ideal of curves of genus $2$
on rational normal scrolls}
\author{Andrea Hofmann}
\address{Matematisk Institutt\\
Universitetet I Oslo\\
PO Box 1053\\
Blindern, NO-0316 Oslo, Norway}
\email{andrehof@math.uio.no}

\begin{document}

\begin{abstract}
Given a smooth curve of genus $2$ embedded in $\mathbf P^{d-2}$ with a complete linear system of degree $d\geq 6$, we list all types of rational normal scrolls arising from linear systems $g^1_2$ and $g^1_3$ on $C$.

\noindent Furthermore, we give a description of the ideal of such a curve of genus $2$ embedded in $\mathbf P^{d-2}$ as a sum of the ideal of the two-dimensional scroll defined by the unique $g^1_2$ on $C$ and the ideal of a three-dimensional scroll arising from a $g^1_3$ on $C$ and not containing the scroll defined by the $g^1_2$ on $C$.
\end{abstract}

\maketitle

\section{Introduction}\label{intro}

\noindent Given a projective variety $X$, in order to find a description of its ideal $I_X$ which is, in some sense, easy to handle, one useful approach is to look for a decomposition of $I_X$ as a sum of ideals of higher-dimensional varieties that contain $X$.
This problem at hand extends naturally to the question whether the syzygies of a given variety are generated by the syzygies of higher-dimensional varieties containing this variety.\\
In the past decades this question has been studied for curves $C\subseteq \mathbf P^n$. 
Since each $g^1_k$ on $C$ for $k\leq n-1$ gives rise to a rational normal scroll that contains the curve, natural candidates for these higher-dimensional varieties are exactly rational normal scrolls.
The above presented problem has been studied to some extent for elliptic normal curves and canonical curves:

\noindent In 1984 Green (\cite{Green1}) showed that the space of quadrics in the ideal of a non-hyperelliptic canonical curve of genus $g\geq 5$ is spanned by quadrics of rank less or equal to $4$.

\noindent In \cite{vBoHu} v. Bothmer and Hulek prove that the linear syzygies of elliptic normal curves, of their secant varieties and of bielliptic canonical curves are generated by syzygies of rational normal scrolls that contain these varieties.

\noindent In \cite{vBo} v. Bothmer proves that the $j$th syzygies of a general canonical curve of genus $g$ are generated by the $j$th syzygies of rational normal scrolls containing the curve for the cases $(g,j)\in \{(6,1),(7,1),(8,2)\}$,
and in \cite{vBo1} the author shows that the first syzygies of general canonical curves of genus $g\geq 9$ are generated by scrollar syzygies.  

\noindent In this paper we will study curves of genus $2$, which cannot be canonically embedded and are all hyperelliptic, and we will focus on the $0$th syzygies, i.e. the ideal $I_C$.

\noindent Let $C$ be a curve of genus $2$, embedded as a smooth and irreducible curve in $\mathbf P^{d-2}$ with a complete linear system of degree $d\geq 6$. Since the linear system is complete, the embedded curve $C\subseteq \mathbf P^{d-2}$ is linearly normal.\\

\noindent Our main result is the following:

\begin{thm1}
Let $C$ be a smooth and irreducible curve of genus $2$, linearly normal embedded in $\mathbf P^{d-2}$ with a complete linear system $\vert H_C\vert$ of degree $d\geq 6$. Moreover, let $S$ be the $g^1_2(C)$-scroll, and let $V$ be a $g^1_3(C)$-scroll that does not contain $S$. Then we have
$$I_S+I_V=I_C.$$ 
\end{thm1}

\noindent In Section \ref{idealC} we will give a proof of this theorem which is based on an inductive argument.

\noindent In Sections \ref{scrollS} and \ref{scrollV} we will list all possible scroll types of the unique $g^1_2(C)$-scroll $S$ and $g^1_3(C)$-scrolls $V_{\vert D\vert}$. We give a connection between the linear system $\vert H_C\vert$ that embeds the curve into projective space and the scroll types of $S$ and a $V_{\vert D\vert}$ and provide thus a proof of existence in all cases.

\section{Preliminaries}\label{prelim}
\noindent Let $C$ be a non-singular curve of genus $2$, and let $\vert H_C\vert$ be a complete linear system of degree $d\geq 6$ on $C$.
By the Riemann-Roch theorem for curves (see e.g. \cite{Ha}, Thm. 1.3 in Chapter IV.1) the system $\vert H_C\vert$ embeds $C$ into projective space $\mathbf P^{d-2}$.
Since $\vert H_C\vert$ is complete, the embedded curve $C\subseteq \mathbf P^{d-2}$ is linearly normal. 
Throughout the whole article $C$ will always denote a smooth curve of degree $d\geq 6$ in $\mathbf P^{d-2}$.

\noindent We use the notation $g^1_k(C)$ to denote a $g^1_k$ on $C$. We are interested in rational normal scrolls that arise from the unique $g^1_2(C)$ and from $g^1_3(C)$'s.
The following proposition states that indeed there is only one $g^1_2$ on $C$, and furthermore it computes the dimension of the family of $g^1_3$'s on $C$, which we will denote by $G^1_3(C)$:

\begin{prop}\label{linsystem}
There exists exactly one $g^1_2(C)$, and this is equal to the canonical system $\vert K_C\vert$.
The family $G^1_3(C):=\{g^1_3(C)'s\}$ is two-dimensional.
\end{prop}

\begin{proof}
We use the Riemann-Roch theorem for curves (see e.g. \cite{Ha}, Thm. 1.3 in Chapter IV.1):

\noindent If $D$ is a divisor of degree $2$, then

$$
h^0(\mathcal O_C(D))=1+h^0(\mathcal O_C(K_C-D))=
\left\{\begin{array}{ccc}
1&\textrm{if}&D\notin\vert K_C\vert,\\
2&\textrm{if}&D\in \vert K_C\vert.\\
\end{array}\right.
$$

\noindent Hence we can conclude that the linear system $\vert D\vert$ is a $g^1_2(C)$ if and only if $\vert D\vert=\vert K_C\vert$.

\noindent If $D$ is a divisor of degree $3$ on $C$, then, again by the Riemann-Roch theorem for curves, $h^0(\mathcal O_C(D))=2$, i.e. each linear system $\vert D\vert$ of degree $3$ is a $g^1_3(C)$.
The set of all effective divisors of degree $3$ on $C$ is given by $C_3:=(C\times C\times C)/S_3$, where $S_3$ denotes the symmetric group on $3$ letters. 
The dimension of this family is equal to $3$, and since each linear system $\vert D\vert$ of degree $3$ has dimension $1$, as shown above, the family of $g^1_3(C)$'s has to be two-dimensional.
\end{proof}

\noindent We are interested in the rational normal scroll arising from the unique $g^1_2(C)$, and for each $\vert D\vert$ in the two-dimensional family $G^1_3(C)$ we are interested in the scroll defined by $\vert D\vert$.

\noindent There are several different presentations of a rational normal scroll. We will use two of these, which will be given in the following paragraphs.

\begin{defn}\label{ratscroll}\textup{(cf. \cite{Schreyer})}\\
\noindent Let $e_1,e_2,\ldots,e_k$ be integers with $e_1\geq e_2\geq \cdots \geq e_k\geq 0$ and $e_1+e_2+\cdots +e_k\geq 2$. Set $\mathcal E=\mathcal O_{\mathbf P^1}(e_1)\oplus \mathcal O_{\mathbf P^1}(e_2)\oplus \cdots \oplus \mathcal O_{\mathbf P^1}(e_k)$, a locally free sheaf of rank $k$ on $\mathbf P^1$, and let
$\pi: \mathbf P(\mathcal E)\to \mathbf P^{1}$ be the corresponding $\mathbf P^{k-1}$-bundle.   

\noindent A \textup{rational normal scroll} $X$ is the image of the map $\iota:\mathbf P(\mathcal E)\hookrightarrow \mathbf P^N:=\mathbf P H^0(\mathbf P(\mathcal E), \mathcal O_{\mathbf P(\mathcal E)}(1))$.\\
The \textup{scroll type} of $X$ is defined to be equal to $(e_1,e_2,\cdots ,e_k)$.
\end{defn}

\begin{rem}\label{degscroll}
\noindent The dimension of $X$ is equal to $k$, and the degree of $X$ is equal to the degree of $\mathcal E$ which is equal to $f:=\sum_{i=1}^k{e_i}$.
Moreover, by the Riemann-Roch theorem for vector bundles, $h^0(\mathbf P(\mathcal E),\mathcal O_{\mathbf P(\mathcal E)}(1))=h^0(\mathbf P^1,\mathcal E)=\rk(\mathcal E)+\deg(\mathcal E)=k+\sum_{i=1}^{k}e_i$, i.e. the dimension of the ambient projective space is equal to $N=k+\sum_{i=1}^{k}e_i-1$.\\
Thus for a rational normal scroll $X$ we obtain $\dim(X)+\deg(X)=k+\sum_{i=1}^ke_i=N+1$, and consequently a rational normal scroll $X\subseteq \mathbf P^N$ is a non-degenerate irreducible variety of minimal degree $f=\codim(X)+1$.

\noindent The scroll $X$ is smooth if and only if all $e_i$, $i=1,\ldots, k$, are positive.
If this is the case, then $\iota:\mathbf P(\mathcal E)\to X\subseteq \mathbf P^N$ is an isomorphism.
\end{rem}

\begin{prop}\label{linscroll}
Each linearly normal scroll $X$ over $\mathbf P^1$ is a rational normal scroll.
\end{prop}

\begin{proof}
If $X$ is a linearly normal scroll over $\mathbf P^1$, then
$X=\iota(\mathbf P(\mathcal E))$, where $\mathcal E=\pi_{\ast}\mathcal O_{\mathbf P(\mathcal E)}(1)$ is a vector bundle over $\mathbf P^1$ and $\iota:\mathbf P(\mathcal E)\hookrightarrow \mathbf P(H^0(\mathcal E))$.
By Grothendieck's splitting theorem (cf. \cite{Haze}) every vector bundle over $\mathbf P^1$ splits, i.e. $\mathcal E$ is of the form $\mathcal E=\oplus_{i}\mathcal O_{\mathbf P^1}(e_i)$.
\end{proof}

\noindent A more geometric description of a rational normal scroll is given by the following definition (cf. \cite{Stevens}):

\begin{defn}
Let $e_1,e_2,\ldots, e_k$ be integers with $e_1\geq e_2\geq \ldots \geq e_k\geq 0$ and $\sum_{i=1}^k{e_i}\geq 2$. Let for $i=1,\ldots, k$, $\phi_i: \mathbf P^1\to C_i\subseteq \mathbf P^{e_i}\subseteq \mathbf P^N$, where $N=\sum_{i=1}^k{e_i}-k-1$, parametrize a rational normal curve of degree $e_i$, such that $\mathbf P^{e_1},\ldots,\mathbf P^{e_k}$ span the whole $\mathbf P^N$.
Then $$X=\overline{\bigcup_{P\in \mathbf P^1}{\langle\phi_1(P),\ldots, \phi_k(P)\rangle}}$$
is a \textup{rational normal scroll} of dimension $k$, degree $e_1+\cdots +e_k$ and scroll type $(e_1,\ldots, e_k)$.
In other words, each fiber of $X$ is spanned by $k$ points where each of these lies on a different rational normal curve. We call these $k$ rational normal curves $C_i$ \textup{directrix curves} of the scroll.
\end{defn}

\subsection{The Picard group of rational normal scrolls}\label{Picardscroll}

\noindent Let $H=[\iota^{\ast}\mathcal O_{\mathbf P^N}(1)]$ denote the hyperplane class, and let $F=[\pi^{\ast}\mathcal O_{\mathbf P^1}(1)]$ be the class of a fiber of $\mathbf P(\mathcal E)$. In the following we will use $H$ and $F$ to denote both the classes and divisors in the respective classes.

\noindent The \textit{Picard group} of $\mathbf P(\mathcal E)$ is generated by $H$ and $F$:
$$\Pic(\mathbf P(\mathcal E))=\mathbf Z[H]\oplus \mathbf Z[F].$$ 

\noindent We have the following intersection products:

$$H^k=f=\sum_{i=1}^k{e_i}, \quad H^{k-1}.F=1, \quad F^2=0.$$

\noindent A minimal section of $\mathbf P(\mathcal E)$ is given by
$B_0=H-rF$, where $r\in \mathbf N$ is maximal such that $H-rF$ is effective, in other words, $B_0=H-e_1F$.

\subsection{The $g^1_2(C)$-scroll and $g^1_3(C)$-scrolls}

\noindent By Proposition \ref{linsystem} there exists exactly one $g^1_2$ on $C$. We set

$$
S=\overline{\bigcup_{E\in g^1_2(C)}{\Span(E)}}\subseteq \mathbf P^{d-2},
$$

\noindent where $\Span(E)$ denotes the line in $\mathbf P^{d-2}$ spanned by the two points in the divisor $E$.

\noindent For each $\vert D\vert$ which is a $g^1_3(C)$ we set

$$
V_{\vert D\vert}=\overline{\bigcup_{D^{\prime}\in \vert D\vert}{\Span(D^{\prime})}}\subseteq \mathbf P^{d-2},
$$

\noindent where $\Span(D^{\prime})$ denotes the plane in $\mathbf P^{d-2}$ spanned by the three points in the divisor $D^{\prime}$.

\begin{prop}\label{normal}
Let $C\subseteq \mathbf P^{d-2}$ be a smooth linearly normal curve of degree $d\geq 6$ and genus $2$, let $S$ be the $g^1_2(C)$-scroll, and for a linear system $\vert D\vert$ which is a $g^1_3(C)$ let $V_{\vert D\vert}$ be the $g^1_3(C)$-scroll associated to $\vert D\vert$.
The scrolls $S$ and $V_{\vert D\vert}$ are rational normal scrolls.
\end{prop}

\begin{proof}
The rationality of $S$ and each $V_{\vert D\vert}$ is obvious. 
For the rest notice that if a scroll $X$ contains a linearly normal curve $C$, then also $X$ has to be linearly normal: If $X$ was the image of a non-degenerate variety in higher-dimensional projective space under some projection, then $C$ had to be as well.
We conclude that since $C$ is linearly normal, $S$ and all $V_{\vert D\vert}$ are linearly normal.  
By Proposition \ref{linscroll} we can conclude that $S$ and all $V_{\vert D\vert}$ are rational normal scrolls.
\end{proof}

\noindent Note that the dimension of $S$ is equal to
$\dim(\vert K_C\vert)+\dim(\Span(E))=2$ and that the dimension of $V_{\vert D\vert}$ is equal to $\dim(\vert D\vert)+\dim(\Span(D^{\prime}))=3$.
By Proposition \ref{normal} the scrolls $S$ and $V_{\vert D\vert}$ are rational normal scrolls which implies by the observations in Remark \ref{degscroll} that we obtain the following degrees: 

\begin{equation}\label{degreescroll}
\deg (S)=d-3,\qquad \deg (V_{\vert D\vert})=d-4.
\end{equation}

\begin{prop}\label{classesonS}
Let $C\subseteq \mathbf P^{d-2}$ be a curve of genus $2$ and degree $d\geq 6$, and let $\mathcal E$ be a $\mathbf P^1$-bundle such that the image of the map $\iota:\mathbf P(\mathcal E)\to\mathbf P^{d-2}$ is the $g^1_2(C)$-scroll $S$.

\noindent The class of $C$ on $\mathbf P(\mathcal E)$ is equal to $[C]=2H-(d-6)F$.
\end{prop}

\begin{proof}
Write $[C]=aH+bF$ with $a,b\in \mathbf Z$.
Since $[C].F=2$, we obtain $a=2$, and $[C].H=d$ implies that $d=2(d-3)+b$, i.e.
$b=6-d$.
\end{proof}

\begin{prop}\label{classonV}
Let $C\subseteq \mathbf P^{d-2}$ be a smooth curve of genus $2$ and degree $d\geq 6$, and let $\mathcal E$ be a $\mathbf P^1$-bundle such that the image of $\iota:\mathbf P(\mathcal E)\to \mathbf P^{d-2}$ is a $g^1_3(C)$-scroll $V$ such that $C$ does not pass through the (possibly empty) singular locus of $V$.\\
The class of $C$ on $\mathbf P(\mathcal E)$ is equal to $[C]=3H^2-2(d-6)H.F$.
\end{prop}

\begin{proof}
Since $C$ is of codimension $2$ on $\mathbf P(\mathcal E)$, we can write the class of $C$ on $\mathbf P(\mathcal E)$ as $[C]=aH^2+bH.F$ with $a,b\in \mathbf Z$.
Since $[C].F=3$, we obtain $a=3$, and $[C].H=d$ implies that $d=3(d-4)+b$, i.e.
$b=2(6-d)$.
\end{proof}

\noindent Theorem \ref{idealSV} states that $I_C$ is generated by the union of $I_S$ and the ideal of \textit{one} $g^1_3(C)$-scroll $V_{\vert D\vert}$ that obviously does not contain $S$. Since the ideal of each rational normal scroll is generated by quadrics, in order to give this statement a sense we have to prove that the ideal $I_C$ is generated by quadrics as well:

\begin{thm}\label{genbyquadrics}
Let $C$ be a smooth curve of genus $2$ embedded in $\mathbf P^{d-2}$ with a complete linear system of degree $d\geq 6$. 
The ideal $I_C$ is generated by quadrics.
\end{thm}

\begin{proof}
This is Theorem (4.a.1) in \cite{Green}.
\end{proof}

\begin{cor}\label{notrisecant}
Let $C\subseteq \mathbf P^{d-2}$ be a smooth curve of genus $2$ embedded with a complete linear system of degree $d\geq 6$.
Then $C$ has no trisecant lines.
\end{cor}

\begin{proof}
Since the ideal of $C$ is generated by quadrics, we can write $C=Q_1\cap \ldots \cap Q_r$ where the $Q_i$ are quadrics and $r=h^0(\mathcal I_C(2))$. Any line that intersects $C$ in three points, intersects each quadric $Q_i$ in at least three points, consequently it is contained in each $Q_i$ and hence in the intersection of all $Q_i$'s which is equal to $C$.
\end{proof}

\section{Scroll types of the $g^1_2(C)$-scroll $S$}\label{scrollS}

\noindent In this section we will list all possible types of the $g^1_2(C)$-scroll $S$ and give a connection between the complete linear system $\vert H_C\vert$ that embeds the curve $C$ into projective space and the scroll type of $S$. In this way we also give a proof of existence in all cases.

\noindent First we give a relation between the degrees of the directrix curves of the $g^1_2(C)$-scroll:

\begin{prop}\label{scrolltypeS}
Let $C\subseteq \mathbf P^{d-2}$ be a smooth curve of degree $d\geq 6$ and genus $2$.
For the scroll type $(e_1,e_2)$ of the $g^1_2(C)$-scroll $S$ we have
$e_1-e_2\leq 3$.
\end{prop}

\begin{proof}
\noindent Let $C_0=H-e_1F$ be a minimal section as described as $B_0$ in general in Section \ref{Picardscroll}. Since $C$ and $C_0$ are effective and $C$ is smooth, so $C_0\not\subseteq C$, we have $[C].C_0\geq 0$, which means by Proposition \ref{classesonS} that
$(2H-(d-6)F).(H-e_1F)\geq 0$, consequently $2e_1+2e_2-2e_1-(d-6)\geq 0$.
Since $d=e_1+e_2+3$ the result follows.
\end{proof}

\noindent Now we will describe the relation between $\vert H_C\vert$ with respect to $\vert K_C\vert$ and the scroll type of the $g^1_2(C)$-scroll $S$. By Equation (\ref{degreescroll}) the degree of $S$ is equal to $d-3$.

\subsection{The case when the degree $d$ of the curve $C$ is even}

\noindent If the degree $d$ of the embedded curve $C\subseteq \mathbf P^{d-2}$ is even, then by Proposition \ref{scrolltypeS} the $g^1_2(C)$-scroll $S$ has scroll type $\left(\frac{d-2}{2},\frac{d-4}{2}\right)$ or $\left(\frac{d}{2},\frac{d-6}{2}\right)$. The linear system $\vert H_C-\frac{d-2}{2}K_C\vert$ is of degree $2$ and thus non-empty by the Riemann-Roch theorem for curves and either equal to $\vert K_C\vert$ or equal to $\vert P+Q\vert$, where $P$ and $Q$ are points on $C$ such that $P+Q\notin \vert K_C\vert$. 

\noindent There is the following relation between $\vert H_C-\frac{d-2}{2}K_C\vert$ and the scroll type of $S$:

\begin{prop}\label{scrolltypeS1}
The scroll type of the $g^1_2(C)$-scroll $S$ is equal to $\left(\frac{d-2}{2},\frac{d-4}{2}\right)$
if and only if $\vert H_C-\frac{d-2}{2}K_C\vert=\vert P+Q\vert$, where $P$ and $Q$ are points on $C$ such that $P+Q\notin \vert K_C\vert$.
\end{prop}

\begin{proof}
If the scroll type of $S$ is equal to $\left(\frac{d-2}{2},\frac{d-4}{2}\right)$, then a minimal section $C_0$ is of degree $\frac{d-4}{2}$, and a general hyperplane section of $S$ containing $C_0$ consists of $C_0$ and $\frac{d-2}{2}$ fibers of $S$.
Consequently, $\vert H_C\vert=\vert \frac{d-2}{2}K_C+P+Q\vert$, where $P$ and $Q$ are points in $C_0\cap C$ and $P+Q\notin \vert K_C\vert$.

\noindent Conversely, if $S$ is of scroll type $\left(\frac{d}{2},\frac{d-6}{2}\right)$, a general hyperplane section of $S$ that contains a minimal section $C_0$, which is of degree $\frac{d-6}{2}$, decomposes into $C_0$ and $\frac{d}{2}$ fibers of $S$. Hence $\vert H_C\vert=\frac{d}{2}\vert K_C\vert$.
\end{proof}

\subsection{The case when the degree $d$ of the curve $C$ is odd}

If the degree $d$ of the embedded curve $C$ is odd, then by Proposition \ref{scrolltypeS} the $g^1_2(C)$-scroll $S$ is of scroll type $\left(\frac{d-3}{2},\frac{d-3}{2}\right)$ or $\left(\frac{d-1}{2},\frac{d-5}{2}\right)$.

\noindent The linear system $\vert H_C-\frac{d-3}{2}K_C\vert$ has degree $3$, and it is thus non-empty by the Riemann-Roch theorem for curves. It is either equal to $\vert P+Q+R\vert$, where $P$, $Q$ and $R$ are points on $C$ such that none of $P+Q$, $P+R$ and $Q+R$ is a divisor in $\vert K_C\vert$, or equal to $\vert K_C+P\vert$, where $P$ is a point on $C$.

\noindent The following proposition states the relation between $\vert H_C-\frac{d-3}{2}K_C\vert$ and the scroll type of $S$:

\begin{prop}\label{scrolltypeS2}
The $g^1_2(C)$-scroll $S$ has scroll type $\left(\frac{d-3}{2},\frac{d-3}{2}\right)$ if and only if  $\vert H_C-\frac{d-3}{2}K_C\vert=\vert P+Q+R\vert$, where $P$, $Q$ and $R$ are points on $C$ such that none of $P+Q$, $P+R$ and $Q+R$ is a divisor in $\vert K_C\vert$.
\end{prop}

\begin{proof}
If $S$ has scroll type $\left(\frac{d-3}{2},\frac{d-3}{2}\right)$, then a general hyperplane section of $S$ containing a minimal section $C_0$ is equal to the union of $C_0$ and $\frac{d-3}{2}$ fibers of $S$.
We obtain that $\vert H_C\vert=\vert \frac{d-3}{2}K_C+P+Q+R\vert$, where $P,Q,R$ are points lying on $C_0\cap C$ and none of $P+Q$, $P+R$ or $Q+R$ is a divisor in $\vert K_C\vert$.

\noindent Conversely, if the scroll type of $S$ is equal to $\left(\frac{d-1}{2},\frac{d-5}{2}\right)$, then a general hyperplane section of $S$ that contains $C_0$ decomposes into a minimal section $C_0$ and $\frac{d-1}{2}$ fibers of $S$. 
Consequently, $\vert H_C\vert=\vert\frac{d-1}{2}K_C+P\vert$ where $P$ is a point in $C_0\cap C$.
\end{proof}

\section{Scroll types of $g^1_3(C)$-scrolls $V_{\vert D\vert}$}\label{scrollV}

\noindent In this section we will first give a relation between the degrees of the three directrix curves of a scroll $V_{\vert D\vert}$, and then we will give a connection between $\vert H_C\vert$ and the scroll type of $V_{\vert D\vert}$ for a given $\vert D\vert \in G^1_3(C)$.

\begin{prop}\label{scrolltypeV}
If $V$ is a $g^1_3(C)$-scroll such that the curve $C$ does not intersect the (possibly empty) singular locus of $V$, then for its scroll type $(e_1,e_2,e_3)$ we have $2e_1-e_2-e_3\leq 4$.
\end{prop}

\begin{proof}
Let $B_0=H-e_1F$ denote a minimal section of the bundle $\mathbf P(\mathcal E)$. Since we have $h^0(\mathcal O_V(H-B_0))=h^0(\mathcal O_V(e_1F))=e_1+1\geq 1$, $B_0$ is contained in at least one hyperplane, consequently $B_0$ does not span all of $\mathbf P^{d-2}$. Since $C$ spans all of $\mathbf P^{d-2}$, $B_0$ cannot contain $C$, thus we have that $[C].B_0\geq 0$, i.e. we know that $(3H^2-2(d-6)HF).(H-e_1F)\geq 0$ by Proposition \ref{classonV}, which means that $3e_1+3e_2+3e_3-3e_1-2(d-6)\geq 0$. 
The result follows from $d=e_1+e_2+e_3+4$. 
\end{proof}

\begin{prop}\label{scrolltypesingV}
If $V=V_{\vert D\vert}$ is a singular scroll of scroll type $(e_1,e_2,0)$ such that the curve $C$ intersects the singular locus of $V$, then $e_1$ and $e_2$ satisfy the following: $e_1-e_2\leq 3$.
\end{prop}

\begin{proof}
If $V=V_{\vert D\vert}$ is a singular $g^1_3(C)$-scroll of type $(e_1,e_2,0)$ such that $C$ intersects its singular locus, then a point $P\in \Sing(V)\cap C$ is a basepoint of $\vert D\vert$, i.e. $\vert D\vert=\vert K_C+P\vert$. 
The projection from $P$ maps $C$ to a curve $C^{\prime}$ of degree $d-1$ in $\mathbf P^{d-3}$, and it maps $V_{\vert D\vert}$ to the $g^1_2(C^{\prime})$-scroll of type $(e_1,e_2)$. By Proposition \ref{scrolltypeS} we obtain $e_1-e_2\leq 3$.
\end{proof}

\noindent We will now come to the converse of Proposition \ref{scrolltypesingV}, i.e. the existence part:

\begin{prop}\label{singVexistence}
If $e_1$ and $e_2$ are integers with $e_1\geq e_2\geq 0$, $e_1-e_2\leq 3$ and $e_1+e_2=d-4$ with $d\geq 6$, then there exists a curve $C$ of genus $2$ and a divisor class $\vert H_C\vert$ on $C$ of degree $d$ that embeds $C$ into $\mathbf P^{d-2}$ such that there exists a $g^1_3(C)$-scroll of type $(e_1,e_2,0)$ such that its singular locus intersects the curve $C$.
\end{prop}

\begin{proof}
Let $e_1\geq e_2\geq 0$ be integers with $e_1-e_2\leq 3$ and $e_1+e_2=d-4$. By the results in Section \ref{scrollS} there exists a curve $C$ of genus $2$, embedded with a system $\vert H^{\prime}\vert$ of degree $d-1$ into $\mathbf P^{d-3}$ such that its $g^1_2(C)$-scroll is of type $(e_1,e_2)$. Take a point $P$ on $C$ and reembed the curve $C$ with the linear system $\vert H_C\vert:=\vert H^{\prime}+P\vert$ into $\mathbf P^{d-2}$. The cone over the $g^1_2(C)$-scroll in $\mathbf P^{d-3}$ with $P$ as vertex is a $g^1_3(C)$-scroll in $\mathbf P^{d-2}$ of type $(e_1,e_2,0)$, and the point $P$ lies in the intersection of its singular locus and the curve $C$.
\end{proof}

\noindent Now we will describe all scroll types a $g^1_3(C)$-scroll $V_{\vert D\vert}$ can have as $\vert D\vert$ varies.
We will distinguish between $\vert D\vert$ with one basepoint and $\vert D\vert$ basepoint-free.

\subsection{The case when  $V_{\vert D\vert}$ is smooth or $V_{\vert D\vert}$ is singular and $C$ does not pass through the singular locus of $V_{\vert D\vert}$}

\noindent If $V_{\vert D\vert}$ is smooth or $V_{\vert D\vert}$ is singular, but the curve $C$ does not pass through the singular locus of $V_{\vert D\vert}$, then $\vert D\vert$ necessarily has to be basepoint-free, since a basepoint of $\vert D\vert$ is a point in the intersection of $C$ with the singular locus of $V_{\vert D\vert}$.\\
In order to describe the possible scroll types with a relation to $\vert H_C\vert$ we use a formula given in \cite{Schreyer}. We give an alternative proof of the fact that the following numbers $d_i$ determine the scroll type of a $g^1_3(C)$-scroll $V_{\vert D\vert}$:

\begin{prop}(\cite{Schreyer}, p. 114)\label{di}
\noindent Given a basepoint-free $\vert D\vert\in G^1_3(C)$, set 

\begin{eqnarray*}
d_0&=&h^0(\mathcal O_C(H_C))-h^0(\mathcal O_C(H_C-D)),\\
d_1&=&h^0(\mathcal O_C(H_C-D))-h^0(\mathcal O_C(H_C-2D)),\\
d_2&=&h^0(\mathcal O_C(H_C-2D))-h^0(\mathcal O_C(H_C-3D)),\\
&\vdots&\\
d_{\lfloor \frac{d}{3}\rfloor}&=&h^0(\mathcal O_C(H_C-\lfloor \frac{d}{3}\rfloor D))-\underbrace{h^0(\mathcal O_C(H_C-(\lfloor \frac{d}{3}\rfloor+1)D))}_{=0}.
\end{eqnarray*}

\noindent The scroll type $(e_1,e_2,e_3)$ of $V_{\vert D\vert}$ is then given by

\begin{eqnarray*}
e_1&=&\#\{j\vert d_j\geq 1\}-1,\\
e_2&=&\#\{j\vert d_j\geq 2\}-1,\\
e_3&=&\#\{j\vert d_j\geq 3\}-1.
\end{eqnarray*}
\end{prop}

\begin{proof}
Since $C$ and $V$ are both linearly normal and span all of $\mathbf P^{d-2}$, there is an isomorphism 
$H^0(C,\mathcal O_C(H_C))\cong H^0(\mathbf P(\mathcal E),\mathcal O_{\mathbf P(\mathcal E)}(H))\cong H^0(\mathbf P^1,\mathcal O_{\mathbf P^1}(e_1)\oplus \mathcal O_{\mathbf P^1}(e_2)\oplus \mathcal O_{\mathbf P^1}(e_3))$, and thus we obtain

\vspace{-0.5cm}

\begin{small}
$$H^0(C,\mathcal O_C(H-iD))\cong H^0(\mathbf P(\mathcal E),\mathcal O_{\mathbf P(\mathcal E)}(H-iF))\cong H^0(\mathbf P^1,\mathcal O_{\mathbf P^1}(e_1-i)\oplus \mathcal O_{\mathbf P^1}(e_2-i)\oplus \mathcal O_{\mathbf P^1}(e_3-i))$$
\end{small}

\vspace{-0.cm}

\noindent for all $i \in \mathbf N_0$.

\noindent Consequently we obtain

$$
d_i=\left\{\begin{array}{ccl}
3,&\textrm{if } &0\leq i\leq e_3,\\
2,&\textrm{if } &e_3+1\leq i\leq e_2,\\
1,&\textrm{if } &e_2+1\leq i \leq e_1.\\
\end{array}\right.
$$

\end{proof}

\noindent Using the formula in Proposition \ref{di} we will now give all possible scroll types of $V_{\vert D\vert}$, in relation to $\vert H_C\vert$:

\vspace{0.2cm}

\begin{center}
{\renewcommand{\arraystretch}{1.5}\renewcommand{\tabcolsep}{0.2cm}
\begin{tabular}{|c|c|c|}
\hline
$d$&Conditions on $\vert H_C\vert$ &Scroll type of $V_{\vert D\vert}$\\
\hline\hline
\multirow{2}{*}{$d\equiv_{(3)}0$}
&
$\vert H_C-\frac{d-3}{3} D\vert=\vert D\vert$
&
$(\frac{d}{3},\frac{d}{3}-2,\frac{d}{3}-2)$
\\
\cline{2-3}
&
$\vert H_C-\frac{d-3}{3} D\vert \neq \vert D\vert$
&
$(\frac{d}{3}-1,\frac{d}{3}-1,\frac{d}{3}-2)$\\
\hline\hline
\multirow{2}{*}{$d\equiv_{(3)}1$}
&
$\vert H_C-\frac{d-1}{3}D\vert \neq \emptyset$
&
$(\frac{d-1}{3},\frac{d-1}{3}-1,\frac{d-1}{3}-2)$\\
\cline{2-3}
&
$\vert H_C-\frac{d-1}{3}D\vert=\emptyset$
&
$(\frac{d-1}{3}-1,\frac{d-1}{3}-1,\frac{d-1}{3}-1)$\\
\hline\hline
\multirow{2}{*}{$d\equiv_{(3)}2$}
&
$\vert H_C-\frac{d-2}{3} D\vert=\vert K_C\vert$
&
$(\frac{d-2}{3},\frac{d-2}{3},\frac{d-2}{3}-2)$\\
\cline{2-3}
&
$\vert H_C-\frac{d-2}{3} D\vert=\vert P+Q\vert$,
&
$(\frac{d-2}{3},\frac{d-2}{3}-1,\frac{d-2}{3}-1)$
\\
&
$P+Q\notin \vert K_C\vert$&\\
\hline
\end{tabular}}
\end{center}

\vspace{0.15cm}

\subsection{The case when $V_{\vert D\vert}$ is singular and $C$ passes through the singular locus of $V_{\vert D\vert}$}

\noindent If $V_{\vert D\vert}$ is singular, then its scroll type is equal to $(e_1,e_2,0)$ where $e_1\geq e_2\geq 0$. 
By Proposition \ref{scrolltypesingV}, if $C$ passes through the singular locus of $V_{\vert D\vert}$, then $e_2=0$ is only possible for $d\leq 7$.
For $d\geq 8$ the scroll $V_{\vert D\vert}$ has exactly one singular point which we will denote by $P$. For $d=6$ and $d=7$ let $P$ denote one point in the singular locus of $V_{\vert D\vert}$.
If the system $\vert D\vert$ has a basepoint $P$, then the curve passes through the singular locus of the scroll $V_{\vert D\vert}$. As suggested in Proposition \ref{singVexistence}, given $\vert H_C\vert$ and $\vert D\vert=\vert K_C+P\vert$, we can find the scroll type of $V_{\vert D\vert}$ by projecting from the point $P$ and using Proposition \ref{scrolltypeS1} and Proposition \ref{scrolltypeS2} for the $g^1_2(C^{\prime})$-scroll, where $C^{\prime}$ is the image of $C$ under the projection. Let in this situation $P^{\prime}$ always denote the point in $\vert K_C-P\vert$.

\subsubsection{The case when the degree $d$ of the curve $C$ is even}
If $d$ is even, then, since each linear system of degree $2$ is non-empty by the Riemann-Roch theorem for curves, we can write $\vert H_C\vert=\vert \frac{d}{2}K_C\vert=\vert \frac{d-2}{2}K_C+P+P^{\prime}\vert$ or $\vert H_C\vert=\vert \frac{d-2}{2}K_C+Q_1+Q_2\vert=\vert \frac{d-4}{2}K_C+Q_1+Q_2+P+P^{\prime}\vert$, where $Q_1$ and $Q_2$ are points on $C$ such that $Q_1+Q_2\notin \vert K_C\vert$.

\begin{prop}\label{scrolltypeV1}
Given $\vert D\vert=\vert K_C+P\vert$, the scroll type of $V_{\vert D\vert}$ is equal to 
$\left(\frac{d-4}{2},\frac{d-4}{2},0\right)$ if and only if 
$\vert H_C-\frac{d-2}{2}K_C-P\vert=\emptyset$.
\end{prop}

\begin{proof}
Let $\vert H_C-\frac{d-2}{2}K_C-P\vert=\emptyset$. Projecting $C$ from $P$ yields a curve $C^{\prime}$ of genus $2$ and degree $d-1$ which is embedded in $\mathbf P^{d-3}$ with the linear system $\vert H^{\prime}\vert:=\vert H_C-P\vert$. Under this projection the scroll $V_{\vert D\vert}$ maps to the $g^1_2(C^{\prime})$-scroll $S^{\prime}$. The scroll $V_{\vert D\vert}$ is thus the cone over $S^{\prime}$ with $P$ as vertex, so if $S^{\prime}$ is of scroll type $(e_1,e_2)$, then $V_{\vert D\vert}$ is of scroll type $(e_1,e_2,0)$.\\
Since $\vert H^{\prime}-\frac{d-2}{2}K_C\vert=\emptyset$, we obtain by Proposition \ref{scrolltypeS2} that the scroll type of $V_{\vert D\vert}$ is equal to $\left(\frac{d-4}{2},\frac{d-4}{2},0\right)$.

\noindent Conversely, if $\vert H_C-\frac{d-2}{2}K_C-P\vert\neq \emptyset$, then the same procedure as above, i.e. projecting from $P$, yields $\vert H^{\prime}-\frac{d-2}{2}K_C\vert\neq \emptyset$, i.e. by Proposition \ref{scrolltypeS2} the scroll type of $V_{\vert D\vert}$ is equal to 
 $\left(\frac{d-2}{2},\frac{d-6}{2},0\right)$.
\end{proof}

\subsubsection{The case when the degree $d$ of the curve $C$ is odd}
If $d$ is odd, then we can write either $\vert H_C\vert=\vert \frac{d-1}{2}K_C+Q\vert=\vert \frac{d-3}{2}K_C+Q+P+P^{\prime}\vert$ or $\vert H_C\vert=\vert \frac{d-3}{2}K_C+\sum_{i=1}^3Q_i\vert=\vert \frac{d-5}{2}K_C+\sum_{i=1}^3Q_i+P+P^{\prime}\vert$, where $Q_1$, $Q_2$ and $Q_3$ are points on $C$ such that $Q_i+Q_j\notin \vert K_C\vert$ for all $i,j\in \{1,2,3\}$, $i\neq j$.

\begin{prop}\label{scrolltypeV2}
The scroll type of the $g^1_3(C)$-scroll $V_{\vert D\vert}$ associated to $\vert D\vert=\vert K_C+P\vert$ is equal to $\left(\frac{d-3}{2},\frac{d-5}{2},0\right)$ if and only if $\vert H_C-\frac{d-1}{2}K_C-P\vert=\emptyset$.
\end{prop}

\begin{proof}
We use the same idea as in the proof of Proposition \ref{scrolltypeV1}.
Let $\vert H_C-\frac{d-1}{2}K_C-P\vert=\emptyset$. Projecting $C$ from $P$ yields a curve $C^{\prime}$ of genus $2$ in $\mathbf P^{d-3}$, embedded with the system $\vert H^{\prime}\vert:=\vert H_C-P\vert$.
Since $\vert H^{\prime}-\frac{d-1}{2}K_C\vert=\emptyset$, by Proposition \ref{scrolltypeS1} the scroll type of $V_{\vert D\vert}$ is equal to
$\left(\frac{d-3}{2},\frac{d-5}{2},0\right)$.

\noindent Conversely, if $\vert H_C-\frac{d-1}{2}K_C-P\vert\neq \emptyset$, then $\vert H^{\prime}-\frac{d-1}{2}K_C\vert\neq\emptyset$.
By Proposition \ref{scrolltypeS2} we obtain that the scroll type of $V_{\vert D\vert}$ is equal to $\left(\frac{d-1}{2},\frac{d-7}{2},0\right)$.
\end{proof}

\begin{exa}
\noindent In order to illustrate Propositions \ref{di} and \ref{scrolltypeV2} we study the case $d=7$, i.e. let now $C\subseteq \mathbf P^5$ be a smooth curve of genus $2$ and degree $7$.
Below we list all scroll types of a $g^1_3(C)$-scroll $V_{\vert D\vert}$ in relation to $\vert H_C\vert$. In our description we will consider the systems $\vert H_C-3K_C\vert$ and $\vert H_C-2D\vert$. Both are of degree $1$, and thus each of these can either be empty or consist of one point.

\vspace{0.2cm}

\begin{small}
\begin{center}
$$
\begin{array}{|l|l|c|c|c|}
\hline
\begin{array}{l}
\textrm{Conditions on}\\
\textrm{the basepoint}\\
\textrm{locus of }\vert D\vert\\
\end{array}
 &
\begin{array}{c}
\hspace{1.35cm} \vert H_C\vert \\
\end{array}
&
\begin{array}{c}
\vert H_C-2D\vert\\
\end{array}
&
\begin{array}{c}
\vert H_C-3K_C\vert \\
\end{array}
& 
\begin{array}{c}
\textrm{Scroll type} \\
\textrm{of }V_{\vert D\vert}\\
\end{array}
\\
\hline\hline
\begin{array}{l}
\textrm{No basepoints}\\
\end{array}
&
\begin{array}{c}
\vert D+2K_C\vert\\
\end{array}
&
\begin{array}{c}
\emptyset \\
\end{array}
&
\begin{array}{c}
\emptyset \\
\end{array}
&
\begin{array}{c}
(1,1,1)\\
\end{array}\\
\hline
\begin{array}{l}
\textrm{No basepoints}\\
\end{array}
&
\begin{array}{l}
\vert D+K_C+Q_1+Q_2\vert,\\
Q_1+Q_2\notin \vert K_C\vert,\\
\vert D-Q_1-Q_2\vert =\emptyset,\\
\vert D-Q_1^{\prime}-Q_2^{\prime}\vert=\emptyset,\\
\textrm{where}\\
\vert Q_1^{\prime}\vert=\vert K_C-Q_1\vert,\\
\vert Q_2^{\prime}\vert=\vert K_C-Q_2\vert\\
\end{array}
&
\begin{array}{c}
\emptyset \\
\end{array}
&
\begin{array}{c}
\emptyset \\
\end{array}
&
\begin{array}{c}
(1,1,1)\\
\end{array}\\
\hline
\begin{array}{l}
\textrm{No basepoints}\\
\end{array}
&
\begin{array}{l}
\vert D+K_C+Q_1+Q_2\vert,\\
Q_1+Q_2\notin \vert K_C\vert,\\
\vert D-Q_1-Q_2\vert =\emptyset,\\
\vert D-Q_1^{\prime}-Q_2^{\prime}\vert\neq\emptyset,\\
\textrm{where}\\
\vert Q_1^{\prime}\vert=\vert K_C-Q_1\vert,\\
\vert Q_2^{\prime}\vert=\vert K_C-Q_2\vert\\
\end{array}
&
\begin{array}{c}
\emptyset \\
\end{array}
&
\begin{array}{c}
\vert D-Q_1^{\prime}-Q_2^{\prime}\vert\\
(\textrm{non-empty}) \\
\end{array}
&
\begin{array}{c}
(1,1,1)\\
\end{array}\\
\hline
\end{array}
$$
\end{center}
\end{small}

\begin{small}
\begin{center}
$$
\begin{array}{|c|l|c|c|c|}
\hline
\begin{array}{l}
\textrm{Conditions on}\\
\textrm{the basepoint}\\
\textrm{locus of }\vert D\vert\\
\end{array}
 &
\begin{array}{c}
\hspace{1.625cm} \vert H_C\vert \\
\end{array}
&
\begin{array}{c}
\vert H_C-2D\vert\\
\end{array}
&
\begin{array}{c}
\vert H_C-3K_C\vert \\
\end{array}
& 
\begin{array}{c}
\textrm{Scroll type} \\
\textrm{of }V_{\vert D\vert}\\
\end{array}
\\
\hline\hline
\begin{array}{l}
\textrm{No basepoints}\\
\end{array}
&
\begin{array}{l}
\vert D+K_C+Q_1+Q_2\vert,\\
Q_1+Q_2\notin \vert K_C\vert,\\
\vert D-Q_1-Q_2\vert \neq\emptyset,\\
\vert D-Q_1^{\prime}-Q_2^{\prime}\vert=\emptyset,\\
\textrm{where}\\
\vert Q_1^{\prime}\vert=\vert K_C-Q_1\vert,\\
\vert Q_2^{\prime}\vert=\vert K_C-Q_2\vert\\
\end{array}
&
\begin{array}{c}
\textrm{non-empty}\\
\end{array}
& 
\begin{array}{c}
\emptyset\\
\end{array}
 &
\begin{array}{c}
 (2,1,0)\\
\end{array}\\
\hline
\begin{array}{l}
\textrm{No basepoints}\\
\end{array}
&
\begin{array}{l}
\vert D+K_C+Q_1+Q_2\vert,\\
Q_1+Q_2\notin \vert K_C\vert,\\
\vert D-Q_1-Q_2\vert \neq\emptyset,\\
\vert D-Q_1^{\prime}-Q_2^{\prime}\vert\neq\emptyset,\\
\textrm{where}\\
\vert Q_1^{\prime}\vert=\vert K_C-Q_1\vert,\\
\vert Q_2^{\prime}\vert=\vert K_C-Q_2\vert\\
\end{array}
&
\begin{array}{c}
\textrm{non-empty}\\
\end{array}
& 
\begin{array}{c}
\vert D-Q_1^{\prime}-Q_2^{\prime}\vert\\
(\textrm{non-empty})\\
\end{array}
 &
\begin{array}{c}
 (2,1,0)\\
\end{array}\\
\hline
\begin{array}{l}
\textrm{One basepoint }P\\
\end{array}
&
\begin{array}{l}
\vert D+K_C+Q_1+Q_2\vert\\
=\vert 2K_C+P+Q_1+Q_2\vert,\\
Q_1+Q_2\notin \vert K_C\vert,\\
P\neq Q_i\\
\textrm{for all }i\in\{1,2\},\\
P+Q_i \notin \vert K_C\vert\\
\textrm{for all } i\in\{1,2\}\\
\end{array}
&
\begin{array}{c}
\emptyset\\
\end{array}
&
\begin{array}{c}
\emptyset\\
\end{array}
&
\begin{array}{c}
(2,1,0)\\
\end{array}
\\
\hline
\begin{array}{l}
\textrm{One basepoint }P\\
\end{array}
&
\begin{array}{l}
\vert D+K_C+Q_1+Q_2\vert\\
=\vert 2K_C+P+Q_1+Q_2\vert,\\
Q_1+Q_2\notin \vert K_C\vert,\\
P\neq Q_i\textrm{ for all}\\
i\in\{1,2\},\\
P+Q_i \in \vert K_C\vert\\
\textrm{for some } i\in\{1,2\}\\
\end{array}
&
\begin{array}{c}
\emptyset\\
\end{array}
&
\begin{array}{l}
\textrm{non-empty},\\
\textrm{different}\\
\textrm{from }P\\
\end{array}
&
\begin{array}{c}
(2,1,0)\\
\end{array}\\
\hline
\begin{array}{l}
\textrm{One basepoint }P\\
\end{array}
&
\begin{array}{l}
\vert 2D+Q\vert\\
=\vert 2K_C+2P+Q\vert,\\
P+Q\notin \vert K_C\vert,\\
2P \notin \vert K_C\vert\\
\end{array}
&
\begin{array}{c}
Q\\
(\textrm{non-empty})\\
\end{array}
&
\begin{array}{c}
\emptyset\\
\end{array}
&
\begin{array}{c}
(2,1,0)\\
\end{array}\\
\hline
\begin{array}{l}
\textrm{One basepoint }P\\
\end{array}
&
\begin{array}{l}
\vert 2D+Q\vert\\
=\vert 2K_C+2P+Q\vert,\\
P+Q\notin \vert K_C\vert,\\
P\neq Q,\\
2P\in \vert K_C\vert\\
\end{array}
&
\begin{array}{c}
Q\\
(\textrm{non-empty})\\
\end{array}
&
\begin{array}{l}
\hspace{0.8cm} Q\\
(\textrm{non-empty},\\
\textrm{different}\\
\textrm{from } P)\\
\end{array}
&
\begin{array}{c}
(2,1,0)\\
\end{array}\\
\hline
\begin{array}{l}
\textrm{One basepoint }P\\
\end{array}
&
\begin{array}{l}
\vert 3K_C+P\vert\\
\end{array}
& 
\begin{array}{c}
\vert K_C-P\vert\\
(\textrm{non-empty})\\
\end{array}
&
\begin{array}{c}
 P\\
\end{array}
& 
\begin{array}{c}
(3,0,0)\\
\end{array}\\
\hline
\end{array}
$$
\end{center}
\end{small}
\end{exa}

\vspace{0.2cm}

\section{The ideal of $C$ as a sum of scrollar ideals}\label{idealC}

\noindent In this section we will show that the ideal $I_C$ of a linearly normal embedded curve $C\subseteq \mathbf P^{d-2}$ of genus $2$ and degree $d\geq 6$ is generated by the ideals $I_S$ and $I_V$, where $S$ is the $g^1_2(C)$-scroll and $V=V_{\vert D\vert}$ is a $g^1_3(C)$-scroll not containing $S$.
In other words, we will prove the following main theorem in this section:

\begin{thm}\label{idealSV}
Let $C$ be a smooth and irreducible curve of genus $2$, linearly normal embedded in $\mathbf P^{d-2}$ with a complete linear system $\vert H_C\vert$ of degree $d\geq 6$. Moreover, let $S$ be the $g^1_2(C)$-scroll, and let $V$ be a $g^1_3(C)$-scroll that does not contain $S$. Then we have
 
$$I_S+I_V=I_C.$$ 
\end{thm}

\noindent We see that in this section we are only interested in $g^1_3(C)$-scrolls $V_{\vert D\vert}$ that do not contain the $g^1_2(C)$-scroll $S$. For this purpose we will now give a criterion for when a given $g^1_3(C)$-scroll $V_{\vert D\vert}$ does not contain the $g^1_2(C)$-scroll $S$:
 
\begin{prop}\label{SinV}
Let $C\subseteq \mathbf P^{d-2}$ be a curve of genus $2$ and degree $d\geq 6$, embedded with the system $\vert H_C\vert$, and let $S$ be the $g^1_2(C)$-scroll.
A $g^1_3(C)$-scroll $V=V_{\vert D\vert}$ contains $S$ if and only if at least one of the following holds:

\begin{itemize}
\item $\vert D\vert$ has a basepoint,
\item $d=6$ and $\vert H_C-D\vert$ has a basepoint or 
\item $d=7$ and $\vert H_C\vert=\vert D +2K_C\vert$.
\end{itemize}
\end{prop}

\begin{proof}
If $\vert D\vert$ has a basepoint $P$, then $\vert D\vert=\vert K_C+P\vert$, hence each fiber of $S$ is contained in a fiber of $V_{\vert D\vert}$, and consequently $V_{\vert D\vert}$ contains $S$.

\noindent Conversely, if $S\subseteq V_{\vert D\vert}$ and $\vert D\vert$ is basepoint-free, then each fiber of $V_{\vert D\vert}$ intersects each fiber of $S$ in one point, since if it did not, then each fiber of $S$ had to be contained in a fiber of $V$ which meant that $\vert D\vert$ had a basepoint.
This implies that each fiber of $V_{\vert D\vert}$, which is a plane, intersects the scroll $S$ in a directrix curve of $S$. This curve is a smooth rational planar curve, consequently the degree of this curve is equal to $1$ or $2$.
This means that, since the degree of $C$ is greater or equal to $6$, the scroll type of $S$ is equal to $(2,1)$ or $(2,2)$, i.e. $d=6$ or $d=7$.
Since each fiber of $V_{\vert D\vert}$ intersects each fiber of $S$ in one point, every divisor in $\vert D+K_C\vert$ spans a $\mathbf P^3$ and thus $h^0(H-(D+K_C))=d-5$, which implies that in the case $d=6$ we obtain that $\vert H-D-K_C\vert$ contains one point, and that in the case $d=7$, $\vert H-D-K_C\vert=\vert K_C\vert$.
\end{proof}

\noindent Before we will prove Theorem \ref{idealSV} we show that $S\cap V=C$ for any $g^1_3(C)$-scroll $V$ that does not contain $S$:

\begin{prop}\label{interSV}
Let $C\subseteq \mathbf P^{d-2}$ be a smooth and irreducible linearly normal curve of genus $2$ and degree $d\geq 6$. For a $g^1_3(C)$-scroll $V=V_{\vert D\vert}$ that does not contain the $g^1_2(C)$-scroll $S$ the following holds:
$$S\cap V=C.$$
\end{prop}

\begin{proof}
Obviously, $C\subseteq S\cap V$. In the case $d=6$ the claim follows by Bézout's Theorem, since then $S\cap V$ is of degree $6$ and dimension $1$. Let now $d\geq 7$. If $S\cap V$ is more than $C$, then it must at least contain one line: If $S\cap V\supseteq C\cup P$ for a point $P$ that does not lie on $C$, then $P$ lies on one fiber $F_0$ of the scroll $S$. But since $P$ does not lie on the curve, each quadric that contains $V$ intersects $F_0$ in at least three points, consequently the whole fiber $F_0$ must be contained in each quadric that contains $V$, and since the ideal $I_V$ is generated by quadrics, $F_0$ is contained in $S\cap V$.
Now there are a priori two possibilities for a fiber $F$ of $S$ to be contained in $V$:
\begin{itemize}
\item[(1)] $F$ is contained in one of the fibers of $V=V_{\vert D\vert}$; this implies that the system $\vert D\vert$ has a basepoint, $\vert D\vert=\vert K_C+P\vert$, and consequently $S\subseteq V$.
\item[(2)] 
$F$ is intersecting each fiber of $V$ in one point. Since $F$ is a fiber of $S$, the point of intersection lies on $C$ for exactly two fibers of $V$. Projecting $C$ from $F$ yields a curve $C^{\prime}$ of genus $2$ and degree $d-2$, linearly normal embedded in $\mathbf P^{d-4}$ with the linear system $\vert H_C-K_C\vert$. The curve $C^{\prime}$ lies on the scrollar surface $S^{\prime}$ which is the image of $V$ under the projection from $F$. A general fiber of $V$ is projected to a fiber in $S^{\prime}$, and the three points in the intersection of $C$ with a general fiber in $V$ are projected to three points on a fiber in $S^{\prime}$, which is impossible by Corollary \ref{notrisecant}, unless $C$ was a curve of degree $7$ in $\mathbf P^5$. If $C$ is a curve of degree $7$ that projects to a curve $C^{\prime}$ of degree $5$ on $\mathbf P^1\times \mathbf P^1$, then $\vert H_C\vert=\vert D+2K_C\vert$, and the $g^1_3(C)$-scroll $V=V_{\vert D\vert}$ contains the $g^1_2(C)$-scroll $S$ by Proposition \ref{SinV}.
\end{itemize}
This proves that the intersection $S\cap V$ cannot contain any line, i.e. in total we obtain $S\cap V=C$.
\end{proof}

\noindent \textit{\textbf{Proof of Theorem \ref{idealSV}}}:\\
\noindent Let $S$ be the $g^1_2(C)$-scroll, and let $V=V_{\vert D\vert}$ be a $g^1_3(C)$-scroll that does not contain $S$.
There is the following short exact sequence of ideal sheaves:
$$0\to \mathcal I_{S\cup V}\to \mathcal I_V\to \mathcal I_{S\cap V}\vert_S\to 0.$$

\noindent By Propositon \ref{interSV} we have $S\cap V=C$, and moreover we know that $\mathcal I_C\vert_S=\mathcal O_S(-C)$. We thus obtain the following short exact sequence:
$$0\to \mathcal I_{S\cup V}\to \mathcal I_V\to \mathcal O_{S}(-C)\to 0.$$

\noindent Tensoring with $\mathcal O_{\mathbf P^{d-2}}(2H)$ and restricting yields the following exact sequence:

\begin{equation*}
0\to \mathcal I_{S\cup V}(2H)\to \mathcal I_{V}(2H)\to \mathcal O_{S}(2H-C)\to 0.
\end{equation*}

\noindent Taking the long exact sequence in cohomology yields

$$
0\to H^0(\mathcal I_{S\cup V}(2))\to H^0(\mathcal I_{V}(2))\to H^0(\mathcal O_{S}(2H-C))\to H^1(\mathcal I_{S\cup V}(2))\to 0.
$$

\noindent Note that $h^1(\mathcal I_V(2))=0$ since $V$ is projectively normal.

\noindent Since $[C]=2H-(d-6)F$ on $S$, we can write the above sequence
in the following form:

\begin{equation*}
0\to H^0(\mathcal I_{S\cup V}(2))\to H^0(\mathcal I_{V}(2))\overset{\psi}{\to} H^0(\mathcal O_{S}((d-6)F))\to H^1(\mathcal I_{S\cup V}(2))\to 0.
\end{equation*}
\\
\noindent Our aim is now to show the following claim:

\begin{claim}\label{psisurjective}
For each $\vert D\vert\in G^1_3(C)=\{g^1_3(C)'s\}$ such that $V_{\vert D\vert}$ does not contain $S$, the map $\psi:  H^0(\mathcal I_{V}(2))\to H^0(\mathcal O_{S}((d-6)F))$ defined via
$$
\psi(Q):=\left\{
\begin{array}{ccc}
0& \textrm{ if } & S\subseteq Q,\\
Q\cap S-C\in \vert (d-6)F\vert&\textrm{ if }& S\not\subseteq Q\\
\end{array}
\right.
$$

\noindent is surjective.
\end{claim}

\noindent If the claim is true, then we have $h^1(\mathcal I_{S\cup V}(2))=0$, and thus the short exact sequence

$$0\to \mathcal I_{S\cup V}(2)\to \mathcal I_{S}(2)\oplus \mathcal I_{V}(2)\to \mathcal I_{\underbrace{S\cap V}_{=C}}(2)\to 0$$

\noindent gives the following long exact sequence in cohomology:

$$0\to H^0(\mathcal I_{S\cup V}(2))\to H^0(\mathcal I_{S}(2))\oplus H^0(\mathcal I_{V}(2))\to H^0(\mathcal I_{C}(2))\to 0.$$

\noindent This implies that

\begin{eqnarray*}
h^0(\mathcal I_C(2))&=&\dim(H^0(\mathcal I_{S}(2))\oplus H^0(\mathcal I_{V}(2)))-h^0(\mathcal I_{S\cup V}(2))\\
&=&\dim(H^0(\mathcal I_{S}(2))+H^0(\mathcal I_{V}(2))).
\end{eqnarray*}

\noindent This argument implies that, since $H^0(\mathcal I_{S}(2))+H^0(\mathcal I_{V}(2))\subseteq H^0(\mathcal I_{C}(2))$,
$$H^0(\mathcal I_{S}(2))+H^0(\mathcal I_{V}(2))= H^0(\mathcal I_{C}(2)),$$

\noindent but since all $I_S$, $I_V$ and $I_C$ are generated by
quadrics, we obtain $I_S+I_V=I_C$.\\

\noindent \textit{Proof of Claim \ref{psisurjective}}:

\noindent Now we will prove by induction that the map 
$$\psi:H^0(\mathcal I_V(2))\to H^0(\mathcal O_S((d-6)F))$$

\noindent as defined above is surjective: \\

\noindent \textit{The induction start: $d=6$ and $d=7$}:\\
For $d=6$ the surjectivity of $\psi$ is obvious. More precisely, if $C\subseteq \mathbf P^4$ is a curve of degree $6$, and if $\vert D\vert$ is a basepoint-free $g^1_3(C)$ such that $\vert H_C-D\vert$ is basepoint-free as well, then by Proposition \ref{SinV} the scroll $V_{\vert D\vert}=:Q_6$ is a quadric that does not contain the $g^1_2(C)$-scroll $S$.\\

\noindent For a curve $C\subseteq \mathbf P^5$ of degree $7$ let $V_{\vert D\vert}$ be a $g^1_3(C)$-scroll that does not contain the $g^1_2(C)$-scroll $S$. For any two quadrics $Q_1\neq Q_2$ in $\mathbf P^5$ their intersection $Q_1\cap Q_2$ is a complete intersection of dimension $3$ and degree $4$, hence if $Q_1$ and $Q_2$ both contained $S$ and $V$, then we have $Q_1\cap Q_2=V\cup \mathbf P^3$.
Since $S\subseteq Q_1\cap Q_2$ and $S$ is irreducible, we must have $S\subseteq V$ or $S\subseteq \mathbf P^3$, but since $S$ spans all of $\mathbf P^5$ and by hypothesis $S$ is not contained in $V$, both cases are impossible.\\
This shows that $h^0(\mathcal I_{S\cup V}(2))\leq 1$, and consequently we obtain
 
$$\dim(H^0(\mathcal I_S(2))+H^0(\mathcal I_V(2)))\geq h^0(\mathcal I_S(2))+h^0(\mathcal I_V(2))-1=8=h^0(\mathcal I_C(2)),$$
and thus $\psi$ is surjective.\\

\noindent \textit{The induction step: $d\geq 8$}:\\
Pick $d-8$ fibers $F_1,\ldots, F_{d-8}$ on $S$.
Let $R_1$ and $R_2$ be two points on $C$ such that $R_1+R_2$ is not a divisor in $\vert K_C\vert$ and such that $\vert R_i\vert\neq \vert H_C-D-2K_C\vert$, $i=1,2$, in the case if $d=8$ and the linear system $\vert H_C-D-2K_C\vert$ is non-empty.
Moreover, let $R_1^{\prime}$ and $R_2^{\prime}$ be two points on $C$ such that $R_1+R_1^{\prime}$ and $R_2+R_2^{\prime}$ are divisors in $\vert K_C\vert$.
Projecting $C$ from the line $L_R$ spanned by $R_1$ and $R_2$ yields a curve $C^{\prime}$ of degree $d-2$ in $\mathbf P^{d-4}$, embedded with the system $\vert H_C-R_1-R_2\vert$. Under this projection the $g^1_2(C)$-scroll $S$ maps to the $g^1_2(C^{\prime})$-scroll $S^{\prime}$, and the scroll $V_{\vert D\vert}$ maps to a $g^1_3(C^{\prime})$-scroll $V_{\vert D\vert}^{\prime}$ that does not contain $S^{\prime}$. 
Notice that the choice of $R_1$ and $R_2$ ensures that $V_{\vert D\vert}^{\prime}$ does not contain $S^{\prime}$ also in the cases $d=8$, $\vert H_C-D-2K\vert\neq \emptyset$ and $d=9$, $\vert H_C\vert=\vert D+3K_C\vert$. (In general it would not have been clear that the system $\vert H_C-R_1-R_2\vert$ does not belong to the cases given in Proposition \ref{SinV}.)\\ 
By the induction hypothesis we find a quadric $Q_{d-2}\subseteq \mathbf P^{d-4}$ which contains $V_{\vert D\vert}^{\prime}$ but not $S^{\prime}$, and which contains the fibers $F_1^{\prime},\ldots, F_{d-8}^{\prime}$, where $F_i^{\prime}$, $i=1,\ldots, d-8$, denotes the image of $F_i$ under the projection.
The cone over $Q_{d-2}$ with the line $L_R$ as vertex is then a quadric $Q_d$ in $\mathbf P^{d-2}$ which contains $V_{\vert D\vert}$ and not $S$. Moreover, $Q_d$ contains the fibers $F_1,\ldots, F_{d-8}$ and two more fibers of $S$:
The fiber $F_{R_1}$ spanned by $R_1$ and $R_1^{\prime}$ intersects the quadric $Q_d$ in three points, counted with multiplicity: The quadric $Q_d$ intersects this line in at least the two points $R_1$ and $R_1^{\prime}$, and since the quadric is singular along the line $L_R$, the quadric $Q_d$ intersects $F_{R_1}$ in the point $R_1$ with at least multiplicity $2$. Consequently $Q_d$ must contain the fiber $F_{R_1}$. The same argument applies to the fiber $F_{R_2}$ which is spanned by the points $R_2$ and $R_2^{\prime}$.\\ 
By degree reasons $Q_d\cap S$ cannot contain more than $C$, $F_{R_1}$, $F_{R_2}$ and $F_1,\ldots, F_{d-8}$. Consequently, $\psi(Q_d)=F_{R_1}\cup F_{R_2}\cup F_1 \cup \ldots \cup F_{d-8}$. Since the divisors $F_{R_1}+F_{R_2}+F_1+\cdots +F_{d-8}$, where $R_1$ and $R_2$ run through all points on $C$ and $F_1,\ldots, F_{d-8}$ run through all fibers of $S$, span the linear system $\vert (d-6)F\vert$, the linearity of $\psi$ together with varying the points $R_1$ and $R_2$ and the fibers $F_1,\ldots, F_{d-8}$ yields the surjectivity of $\psi$.
\qed

\subsection*{Acknowledgements}
\noindent This paper arose from parts of my Ph.D. thesis. I wish to thank my advisor Kristian Ranestad for many interesting discussions and very helpful feedback.


\begin{thebibliography}{99}

\bibitem{Green}
M.L.~Green: Koszul cohomology and the geometry of projective varieties, {\em Journal of Differential Geometry} {\bf 19} (1984) 125--171.

\bibitem{Green1}
M.L.~Green: Quadrics of rank four in the ideal of a canonical curve, {\em Inventiones mathematicae} {\bf 75(1)} (1984) 85--104. 

\bibitem{Ha}
R.~Hartshorne: {\em Algebraic Geometry}, Springer Verlag, 1977

\bibitem{Haze}
M.~Hazewinkel and C.F.~Martin: A short elementary proof of Grothendieck's theorem on algebraic vector bundles over the projective line,  
{\em Journal of Pure and Applied Algebra} {\bf 25(2)} (1982) 207--211.

\bibitem{Schreyer}
F.-O.~Schreyer: Syzygies of canonical curves and special linear series,
{\em Mathematische Annalen} {\bf 275(1)} (1986) 105--137.

\bibitem{Stevens}
J.~Stevens: Rolling factors deformations and extensions of canonical curves,
{\em Documenta Mathematica} {\bf 7} (2002) 185--226.

\bibitem{vBo}
H.-C.~v. Bothmer: Scrollar Syzygies of general canonical curves with genus at most $8$, {\em Trans. Amer. Math. Soc.} {\bf 359} (2007) 465--488. 

\bibitem{vBo1}
H.-C.~v. Bothmer: Geometrische Syzygien von kanonischen Kurven, {\em Dissertation, Universität Bayreuth} (2000)

\bibitem{vBoHu}
H.-C.~v. Bothmer and K.~Hulek: Geometric syzygies of elliptic normal curves and their secant varieties, {\em Manuscripta Mathematica} {\bf 113(1)} (2004) 35--68.
\end{thebibliography}
\end{document}